\documentclass[11pt]{article}

\usepackage{amsfonts}
\usepackage{amsthm}
\usepackage{amssymb, amsmath}
\usepackage[english]{babel}

\newtheorem{theo}{Theorem}[section]

\newtheorem{lemma}[theo]{Lemma}

\theoremstyle{definition}

\newtheorem{remark}[theo]{Remark}

\numberwithin{equation}{section}

\newcommand{\R}{\mathbb{R}}
\newcommand{\N}{\mathbb{N}}

\newcommand{\cN}{\mathcal{N}}
\renewcommand{\dim}{\mathrm{dim}}
\newcommand{\codim}{\mathrm{codim}}
\newcommand{\eps}{\varepsilon}
\newcommand{\vertiii}[1]{{\left\vert\kern-0.25ex\left\vert\kern-0.25ex\left\vert #1 
    \right\vert\kern-0.25ex\right\vert\kern-0.25ex\right\vert}}

\begin{document}
\title{Carl's inequality for quasi-Banach spaces}
\author{Aicke Hinrichs\footnote{Institute of Analysis, University Linz, Altenberger Str. 69, 4040 Linz, Austria, {\tt aicke.hinrichs@jku.at}.},
Anton Kolleck\footnote{Department of Mathematics, Technical University Berlin, Street of 17. June 136, 10623 Berlin, Germany,
{\tt kolleck@math.tu-berlin.de}; this author was supported by the DFG Research Center {\sc Matheon} ``Mathematics for key technologies'' in Berlin.},
Jan Vyb\'\i ral\footnote{Department of Mathematical Analysis, Charles University, Sokolovsk\'a 83, 186 00, Prague 8, Czech Republic, {\tt vybiral@karlin.mff.cuni.cz};
this author was supported by the ERC CZ grant LL1203 of the Czech Ministry of Education and by the Neuron Fund for Support of Science.}}
\date{\today}
\maketitle

\begin{abstract}
 We prove that for any  two quasi-Banach spaces $X$ and $Y$ and any $\alpha>0$ there exists a constant $\gamma_\alpha>0$ such that
 $$ \sup_{1\le k\le n}k^{\alpha}e_k(T)\le \gamma_\alpha \sup_{1\le k\le n} k^\alpha c_k(T) $$
 holds for all linear and bounded operators $T:X\to Y$. Here $e_k(T)$  is the $k$-th entropy number of $T$ and $c_k(T)$  is the $k$-th Gelfand number of $T$.
 For Banach spaces $X$ and $Y$ this inequality is widely used and well-known as Carl's inequality.
 For general quasi-Banach spaces it is a new result.
\end{abstract}

\section{Introduction}
The theory of $s$-numbers \cite{CaSt,Pietsch,Pinkus} (sometimes also called $n$-widths) emerged from the studies of geometry of Banach spaces and of operators
between them but found many applications in numerical analysis as well as linear and non-linear approximation theory \cite{CDD, RD, DL, N, LGM}.
It turned out to be also useful in estimates of eigenvalues of operators \cite{BC, CaTr, Koenig, Pietsch2}.

One of the most useful tools in the study of $s$-numbers is Carl's inequality \cite{BC}, which relates
the behavior of several of the most important scales of $s$-numbers  to their entropy numbers (see below for the exact definitions).
If $X$ and $Y$ are Banach spaces and if $T:X\to Y$ is a bounded linear operator between them,
then Carl's inequality states that for every natural number $n\in\N$ 
\begin{equation}\label{eq:CARL}
\sup_{1\le k\le n}k^{\alpha}e_k(T)\le \gamma_\alpha\sup_{1\le k \le n}k^{\alpha}s_k(T).
\end{equation}
Here, $e_k(T)$ denotes the entropy numbers of $T$ and $s_k(T)$ stands for any of the
approximation, Gelfand, or Kolmogorov numbers. 
For the definition of these quantities, let $T:X\to Y$ be a bounded linear operator between quasi-Banach spaces $X$ and $Y$. Then
we define the Gelfand numbers $c_n(T)$, the Kolmogorov numbers $d_n(T)$, the approximation numbers $a_n(T)$ and the entropy numbers $e_n(T)$, respectively, by
\begin{align*}
   c_n(T) &=\inf\limits_{\substack{M\subset X\\\codim\, M<n}}\sup\limits_{\substack{x\in M\\\|x\|_X\le 1}}\|Tx\|_Y\\
   d_n(T) &=\inf\limits_{\substack{N\subset Y\\\dim N<n}}\sup\limits_{\|x\|_X\le 1}\inf_{z\in N}\|Tx-z\|_Y\\
   a_n(T) &=\inf\{\|T-L\|:L:X\to Y, {\rm rank}(L)<n\}\\
   e_n(T) &=\inf \Bigl\{\varepsilon>0:T(B_X)\subset \bigcup_{j=1}^{2^{n-1}}(y_j+\varepsilon B_Y)\Bigr\}.
\end{align*}
In the last definition, $B_X$ can denote either the open or the closed unit ball in $X$. While usually the closed unit ball is used, 
for technical reasons we prefer to work with the open unit ball $B_X = \{ x\in X \,:\, \|x\|_X < 1\}$.

The main result of this note is that Carl's inequality holds also for quasi-Banach spaces and Gelfand numbers.
\begin{theo}\label{theo:main}
 Let $X$ and $Y$ be quasi-Banach spaces. Then for any $\alpha>0$ there exists a constant $\gamma_\alpha>0$ such that
 \begin{equation}\label{eq:orig1}
  \sup_{1\le k\le n}k^{\alpha}e_k(T)\le \gamma_\alpha \sup_{1\le k\le n} k^\alpha c_k(T)
 \end{equation} 
 holds for all linear and bounded operators $T:X\to Y$.
\end{theo}

Moreover, it was observed already in \cite{BBP,GBach} or \cite[Section 1.3.3]{ET} that Carl's inequality extends
easily to quasi-Banach spaces and Kolmogorov numbers or approximation numbers with only minor modifications necessary.
Consequently, \eqref{eq:CARL} is true also for quasi-Banach spaces with $s_k(T)$ standing again for any of
the scales of approximation, Gelfand, or Kolmogorov numbers, respectively. There is nevertheless one difference
between \eqref{eq:CARL} and \eqref{eq:orig1}. The constant $\gamma_\alpha>0$ in \eqref{eq:CARL} depends
indeed only on $\alpha>0$ and is universal for all pairs of Banach spaces $X$ and $Y$. But in \eqref{eq:orig1},
the constant $\gamma_\alpha>0$ may depend on the quasi-Banach spaces $X$ and $Y$. We give more details in Remark \ref{rem:const}.

We now explain the original proof of Carl's inequality \eqref{eq:orig1}
for Gelfand numbers in the case that $X$ and $Y$ are Banach spaces (cf. \cite{BC}, \cite[Theorem 3.1.1]{CaSt}, or \cite[Theorem 5.2]{Pisier}).
It will become clear, that this approach relies heavily on the Hahn-Banach Theorem applied to both  Banach spaces $X$ and $Y$.
As this fundamental technique fails for general quasi-Banach spaces \cite{K-HB}, this approach cannot be easily transferred.
We refer also to \cite{GK} and \cite{Kalt} for an overview of other results on quasi-Banach spaces and their geometry.

The proof proceeds in the following way.
First, \eqref{eq:orig1} is shown for approximation numbers $a_k(T)$ instead of Gelfand numbers. As noted before, this is easily extended to
the quasi-Banach case. 
Afterwards, an isometric embedding $j:Y \to \ell_\infty(S)$ for some set $S$ is used. Such an embedding exists for any Banach space $Y$ and can be
constructed with the Hahn-Banach theorem, e.g. with $S$ being the unit sphere or the unit ball in the dual space $Y^*$. Already such an isometric
embedding obviously does not exist if $Y$ is not a Banach space. Now, making use of the isometric embedding $j$, the following two properties of entropy and 
Gelfand numbers, namely
\begin{enumerate}
\item[(i)] $e_n(T)\le 2e_n(j\circ T)$ for every $n\in\N$, and 
\item[(ii)] $c_n(T)=a_n(j\circ T)$ for every $n\in\N$
\end{enumerate}
are essential. 
Equipped with these tools, \eqref{eq:orig1} then follows simply by
\begin{equation}\label{eq:orig2}
\sup_{1\le k\le n}k^{\alpha}e_k(T)\le 2 \sup_{1\le k\le n}k^{\alpha}e_k(j\circ T)\le 2\gamma_\alpha \sup_{1\le k\le n} k^\alpha a_k(j\circ T)
\le 2\gamma_\alpha \sup_{1\le k\le n} k^\alpha c_k(T).
\end{equation}
Let us comment on (i) and (ii) - and point out, why (ii) fails completely in the quasi-Banach case.

The proof of (i) is easy. Let $(j\circ T)(B_X)$ be covered by $2^{n-1}$ balls of radius $\varepsilon$ in $\ell_\infty(S)$.
Then (by just the triangle inequality) it can be covered by $2^{n-1}$ balls of radius $2\varepsilon$ in the same space with centers in $(j\circ T)(X).$
Finally, as $j$ is isometric, this can be translated into a covering of $T(B_X)$ by $2^{n-1}$ balls of radius $2\varepsilon$ in $Y$.

The proof of (ii) is more involved - and makes heavy use of the Hahn-Banach theorem. We will only comment on the more difficult inequality $a_n(j\circ T)\le c_n(T)$,
which was used in \eqref{eq:orig2}. The essential property of the space $\ell_\infty(S)$ here is the extension property: any linear bounded operator
$U:M \to \ell_\infty(S)$ from a closed linear subspace $M$ of a Banach space $X$ can be extended to an operator $\tilde{U}:X \to \ell_\infty(S)$ with 
$\|\tilde{U}\|=\|U\|$. Again, the proof of this extension property needs the Hahn-Banach theorem now for the Banach space $X$. So, this step is in general not possible if
$X$ is not a Banach space.

With the extension property the proof of (ii) is finished as follows. Given a subspace $M$ of $X$ with $\codim\, M<n$, we can extend  $U=j\circ T|_M$ to an
operator $\tilde{U}:X \to \ell_\infty(S)$ with $\|\tilde{U}\|=\|U\|=\|T|_M \|$ and, letting $L=j\circ T - \tilde{U}$ we conclude that ${\rm rank}(L)<n$ and
$a_n(j\circ T) \le \|j\circ T - L \| = \|\tilde{U}\| = \|T|_M \|$. Finally, (ii) follows by taking the infimum over all such $M$.
This discussion makes clear, that in this approach to Carl's inequality via the approximation numbers, the property that both $X$ and $Y$ are Banach spaces
(and not merely quasi-Banach spaces) is essential.

From the numerous applications of Carl's inequality available in the literature we will comment on its
role in the area of sparse recovery and compressed sensing \cite{CRT, D}.
It was observed already in the seminal paper of Donoho \cite{D} that lower bounds on the Gelfand numbers
of the embedding $id:\ell_p^N\to\ell_2^N$ with $0<p<1$ can be directly transferred into statements on optimality
of sparse recovery methods. When deriving these lower bounds, Donoho combined Carl's inequality with the known results
on entropy numbers of this embedding, overlooking that Carl's inequality is not available for Gelfand numbers and
quasi-Banach spaces. This flaw was corrected only later in \cite{FPRU} using completely different methods.
Using Theorem \ref{theo:main}, we give an alternative proof of the lower bound of $c_n(id:\ell_p^N\to\ell_2^N)$
in the last section.

The structure of the paper is as follows. In Section \ref{Sec:quasi} we collect some notation and basic facts about quasi-Banach spaces.
Section \ref{Sec:proof} gives the proof of our main result, Theorem \ref{theo:main}. 
With standard arguments, we derive the version of Carl's inequality for Lorentz norms, Theorem \ref{theo:Lorentz}. 
Finally, in Section \ref{CarlGelfand'} we show how to use
Theorem \ref{theo:main} to obtain lower estimates of Gelfand numbers of $id:\ell_p^N\to\ell_q^N$
with implications to optimal recovery of sparse vectors.

\section{Quasi-Banach spaces}\label{Sec:quasi}

This section collects some basic facts about quasi-Banach spaces. We restrict ourselves to the minimum needed later on
and refer to \cite{Kalt} and the references therein for an extensive overview. If $X$ is a (real) vector space, we say that
$\|\cdot\|_X:X\to [0,\infty)$ is a quasi-norm if
\begin{itemize}
\item[(i)] $\|x\|_X=0$ if, and only if $x=0$,
\item[(ii)] $\|\alpha x\|_X=|\alpha|\cdot \|x\|_X$ for all $\alpha\in\R$ and $x\in X$,
\item[(iii)] there is a constant $C\ge 1$, such that $\|x+y\|_X\le C(\|x\|_X+\|y\|_X)$ for all $x,y\in X$. 
\end{itemize}
By the fundamental Aoki-Rolewicz theorem \cite{Aoki,Rol}, every quasi-norm is equivalent to some $p$-norm, i.e. there exists
a mapping $\vertiii{\cdot}_X:X\to [0,\infty)$ and $0<p\le 1$, such that $\vertiii{\cdot}_X$ satisfies (i) and (ii) as above, (iii) gets replaced by
$\vertiii{x+y}_X^p\le \vertiii{x}_X^p+\vertiii{y}_X^p$ and $\vertiii{\cdot}_X$ is equivalent to $\|\cdot\|_X$ on $X$.
The expression $\vertiii{\cdot}_X$ is then called a $p$-norm and $X$ is a $p$-Banach space.
If a quasi-normed vector space $X$ is complete with respect to the metric induced by $\vertiii{\cdot}^p_X$, it is called a quasi-Banach space.
As the validity of Carl's inequality does not change if we replace the quasi-norms on $X$ and $Y$ by equivalent
quasi-norms, we shall always assume that $X$ and $Y$ are equipped with a $p$-norm and a $q$-norm, respectively.

\subsection{Quotients of quasi-Banach spaces}

If $X$ and $Y$ are quasi-Banach spaces and $T$ is a bounded linear operator between them,
we can still define Gelfand numbers $c_n(T)$ as before, as the notion of codimension is algebraic. Furthermore,
if $X$ is a $p$-Banach space and $M\subset X$ is a subspace, we can also define the quotient space $X/M$ and the usual definition
makes it again a $p$-Banach space. Indeed, let $[x],[y]\in X/M$. Then there are (for every $\varepsilon>0$) $z^x,z^y\in M$, such that
$$
\|x-z^x\|_X\le (1+\varepsilon)\|[x]\|_{X/M}\quad \text{and}\quad \|y-z^y\|_X\le (1+\varepsilon)\|[y]\|_{X/M}.
$$
We then obtain
\begin{align*}
\|[x+y]\|^p_{X/M}&\le\|x+y-z^x-z^y\|^p_X\le \|x-z^x\|_X^p+\|y-z^y\|_X^p\\
&\le (1+\varepsilon)^p(\|[x]\|_{X/M}^p+\|[y]\|_{X/M}^p)
\end{align*}
and the statement follows by letting $\varepsilon\to 0.$

\subsection{Entropy numbers of identity mappings}

We give an analogue of \cite[(12.1.13)]{Pietsch} for quasi-Banach spaces.
\begin{lemma} Let $X$ be a real $m$-dimensional $p$-Banach space, where $m\in\N$ and $0<p\le 1$. Then
\begin{equation}\label{eq:qeps1}
e_n(id:X\to X)\le 4^{1/p}2^{-\frac{n-1}{m}}
\end{equation}
for all $n\in \N.$
\end{lemma}
\begin{proof}
The inequality $e_1(id:X\to X)\le 1$ holds also for quasi-Banach spaces. If $(n-1)\le 2m/p$, then $2^{\frac{n-1}{m}}\le 4^{1/p}$ and \eqref{eq:qeps1} follows.

We assume therefore that $(n-1)>2m/p.$ We choose $\varepsilon>0$ by 
\begin{equation*}
\left[\frac{(1+\varepsilon^p/2)^{1/p}}{\varepsilon/2^{1/p}}\right]^m=2^{n-1},\quad \text{i.e.}\quad
\varepsilon=\left[\frac{2}{2^{\frac{p(n-1)}{m}}-1}\right]^{1/p}<1.
\end{equation*}
Let now $x_1,\dots,x_{N}\in B_X$
be a maximal subset of $B_X$ with mutual distances $\|x_i-x_j\|_X\ge \varepsilon.$
Then $B_X$ can be covered by the balls $x_i+\varepsilon B_X$ and the balls $x_i+\frac{\varepsilon}{2^{1/p}}B_X$
are mutually disjoint. Indeed, if there would be a $z\in X$ with $\|x_i-z\|_X< \varepsilon/2^{1/p}$
and $\|x_j-z\|_X< \varepsilon/2^{1/p}$, then $\|x_i-x_j\|^p_X\le \|x_i-z\|^p_X+\|x_j-z\|^p_X<\varepsilon^p.$
Furthermore, if $y\in x_i+\frac{\varepsilon}{2^{1/p}}B_X$, then $y=x_i+z$ with $\|z\|_X<\frac{\varepsilon}{2^{1/p}}$
and
$$
\|y\|_X^p\le\|x_i\|_X^p+\|z\|_X^p< 1+\varepsilon^p/2.
$$
Hence, $x_i+\frac{\varepsilon}{2^{1/p}}B_X$ are mutually disjoint, all included in $(1+\varepsilon^p/2)^{1/p}B_X.$
Comparing the volumes (with respect to any translation invariant normalized measure on $X$), we get
$$
N\cdot \Bigl(\frac{\varepsilon}{2^{1/p}}\Bigr)^m\le (1+\varepsilon^p/2)^{m/p}, \quad \text{i.e.}\quad
N\le \left[\frac{(1+\varepsilon^p/2)^{1/p}}{\varepsilon/2^{1/p}}\right]^m=2^{n-1}.
$$
We therefore obtain that
$$
e_n(id:X\to X)\le \varepsilon=\left[\frac{2}{2^{\frac{p(n-1)}{m}}-1}\right]^{1/p}\le \left[4\cdot 2^{-\frac{p(n-1)}{m}}\right]^{1/p}
$$
and \eqref{eq:qeps1} follows again.

\end{proof}

\section{Proof of Carl's inequality}\label{Sec:proof}

In this section we prove our main result, Theorem \ref{theo:main}, as well as its Lorentz space counterpart,
Theorem \ref{theo:Lorentz}.

\subsection{Proof of Theorem \ref{theo:main}}

Let $X$ be a $p$-Banach space and let $Y$ be a $q$-Banach space.
We fix a sequence $(M_n)_{n\in\N}$ of finite codimensional subspaces of $X$ and let $\tau_n = \big\|T|_{M_n}\big\|$.
We also fix a sequence $(\eps_n)_{n\in\N}$ of positive numbers  with $\eps_n\le 1$.
Let $\mathcal{M}_n\subset X/M_n$ 
be an $\eps_n$-net of the unit ball of $X/M_n$, i.e. for any $[x]\in B_{X/M_n}$ 
there exist $[x_n]\in\mathcal{M}_n$ such that
\begin{equation}\label{eq:quotient1}
\big\|[x]-[x_n]\big\|_{X/M_n}=\inf\limits_{z\in M_n}\|x-x_n-z\|_X<\eps_n.
\end{equation}
As a byproduct of \eqref{eq:quotient1}, we also get
$$
\|[x_n]\|^p_{X/M_n}\le \|[x_n]-[x]\|^p_{X/M_n}+\|[x]\|^p_{X/M_n}<\eps_n^p+1\le 2,
$$
hence $\mathcal{M}_n\subset 2^{1/p}B_{X/M_n}$.
As $\|[x_n]\|_{X/M_n}<2^{1/p}$, we may assume (just by the definition of the quotient quasi-norm) that $x_n\in X$ was chosen with $\|x_n\|_X<2^{1/p}$.
Let $\cN_n\subset 2^{1/p}B_X$ be a lifting of $\mathcal{M}_n$, which collects for each class $[x_n]\in \mathcal{M}_n$
the representative with quasi-norm smaller than $2^{1/p}.$
Using \eqref{eq:quotient1}, we see that for any $x\in B_X$ there exist $x_n\in\cN_n$ and $z_n\in M_n$ with
$$
\| x - (x_n+z_n) \|_X <\eps_n.
$$
Finally, let $\delta_0=1$ and $$ \delta_n = \prod_{j=1}^n \eps_j \qquad \text{for } n\in\N.$$

The proof of Theorem \ref{theo:main} relies on an iterative construction. The single steps are
based on the lifting just described and the details are given in Lemma \ref{lem:qchoice}.
Its inductive use is then the subject of Lemma \ref{lem:qinduction}.

\begin{lemma}\label{lem:qchoice}
 For any $x\in X$ with $\|x\|_X < \delta$ for some $0<\delta\le 1$ there exist $x_n \in \cN_n$ and $z_n\in M_n$ such that
 $$ \|z_n\|_X < 4^{1/p} \qquad \text{and} \qquad \| x - \delta(x_n+z_n) \|_X < \delta\cdot\eps_n.$$
\end{lemma}

\begin{proof}
 Since $\|x/\delta\|_X < 1$, we find $x_n \in \cN_n$ and $z_n\in M_n$ such that
 $$ \| x/\delta - (x_n+z_n) \|_X < \eps_n.$$ This shows the second inequality. The bound on $z_n$ follows from the $p$-triangle inequality
 $$ \|z_n\|^p_X \le \| x/\delta - (x_n+z_n) \|^p_X + \| x/\delta \|_X^p +\|x_n\|_X^p< 4.$$
\end{proof}

\begin{lemma}\label{lem:qinduction}
 For any $x\in B_X$, there exist sequences $(x_n)_{n\in\N},(z_n)_{n\in\N}$ with $x_n \in \cN_n$ and $z_n \in M_n$ such that
 \begin{enumerate}
	 \item[(i)]   $\|z_n\|_X < 4^{1/p}$ for $n\in\N$,
	 \item[(ii)]  $\left\| x - \sum_{k=1}^n \delta_{k-1} (x_k+z_k) \right\|_X < \delta_n$ for $n\in\N$,
	 \item[(iii)] $\left\| Tx - \sum_{k=1}^n \delta_{k-1} T x_k \right\|^q_Y \le (\|T\| \delta_n)^q + 4^{q/p} \sum_{k=1}^n \delta^q_{k-1} \tau^q_k$ for $n\in\N$.
 \end{enumerate}
\end{lemma}

\begin{proof}
 The existence of sequences $(x_n)_{n\in\N},(z_n)_{n\in\N}$ with $x_n \in \cN_n$ and $z_n \in M_n$ satisfying $(i)$ and $(ii)$ follows inductively from Lemma \ref{lem:qchoice}.
 It remains to prove $(iii)$. 
 We use the $q$-triangle inequality, $(i)$,$(ii)$ and $\tau_n = \big\|T|_{M_n}\big\|$ to conclude that
\begin{align*}
   \displaystyle \left\| Tx - \sum_{k=1}^n \delta_{k-1} T x_k \right\|^q_Y 
	  &\le  \left\| Tx - \sum_{k=1}^n \delta_{k-1} T (x_k+z_k) \right\|^q_Y + \sum_{k=1}^n \delta^q_{k-1} \| T z_k\|^q_Y \\
    &\le \big(\|T\| \delta_n\big)^q + 4^{q/p} \sum_{k=1}^n \delta^q_{k-1} \tau^q_k.
 \end{align*}
\end{proof}

The next theorem now follows using the optimal subspaces $M_n$ and the inequality \eqref{eq:qeps1} for entropy numbers for the $m_n$-dimensional $p$-Banach space $X/M_{m_n}$.

\begin{theo}\label{thm:monster}
Let $T:X\to Y$ be a bounded linear operator from the $p$-Banach space $X$ to the $q$-Banach space $Y$, where $0<p,q\le 1$.
Let $(k_n)_{n\in\N}$ and $(m_n)_{n\in\N}$ be sequences of positive integers. Then
$$ e_{k_1+\dots+k_n+1-n} (T)^q \le  2^{2nq/p-\sum_{j=1}^n \frac{k_j-1}{m_j}\cdot q} \|T\|^q +
4^{q/p} \sum_{\ell=1}^n  2^{2(\ell-1) q/p-\sum_{j=1}^{\ell-1} \frac{k_j-1}{m_j}\cdot q} c_{m_\ell+1} (T)^q. $$
\end{theo}
\begin{proof}
We apply Lemma \ref{lem:qinduction} to a sequence $ (M_{m_j})_{j\in\N}$
of finite codimensional subspaces of $X$ with $\codim M_{m_j} = m_j$.
Moreover, by \eqref{eq:qeps1}, we may and do choose the $\eps_{m_j}$-nets 
$\mathcal{M}_{m_j}\subset 2^{1/p} B_{X/M_{m_j}}$ and the lifting 
$\mathcal{N}_{m_j}\subset 2^{1/p} B_X$ such that
$$
	\eps_j=4^{1/p}2^{-\frac{k_j-1}{m_j}} 
	\qquad \text{and} \qquad
	\# \mathcal{M}_{m_j} = \# \mathcal{N}_{m_j} \le 2^{k_j-1}.
$$	
We then have
\begin{equation}\label{eq:delta} 
\delta_k=\prod\limits_{j=1}^{k}\eps_{m_j}=2^{2k/p-\sum\limits_{j=1}^{k}\frac{k_j-1}{m_j}}\qquad \text{for } k\in\N.
\end{equation}
For the net 
$$ {\mathcal N} = \left\{ T \Bigl( \sum_{j=1}^n \delta_{j-1} x_j \Bigr) \,:\, x_j \in {\mathcal N}_{m_j}\right\}$$
we derive
\begin{equation} \label{eq:card} 
\# {\mathcal N} \le \prod_{j=1}^n \# {\mathcal N}_{m_j} \le 2^{\sum_{j=1}^n k_j - n}.
\end{equation}
Now $(iii)$ in Lemma \ref{lem:qinduction} shows that for any $x\in B_X$ there exists $y\in {\mathcal N}$ such that
$$ \|Tx-y\|_Y^q \le \big(\|T\| \delta_n\big)^q + 4^{q/p} \sum_{\ell=1}^n \delta^q_{\ell-1} \tau^q_{m_\ell}. $$
Hence, using \eqref{eq:delta} and \eqref{eq:card}, we obtain
	$$
	e_{k_1+\dots+k_n+1-n} (T)^q \le  2^{2nq/p-\sum_{j=1}^n \frac{k_j-1}{m_j}\cdot q} \|T\|^q +
	4^{q/p} \sum_{\ell=1}^n  2^{2(\ell-1) q/p-\sum_{j=1}^{\ell-1} \frac{k_j-1}{m_j}\cdot q} \tau_{m_\ell}^q
	$$
	and the claim follows by taking the infimum over all sequences of subspaces $(M_{m_j})_{j=1}^n$ with $\codim M_{m_j} \le m_j$.
\end{proof}

We are now ready to prove Theorem \ref{theo:main}. It is enough to show
$$
n^{\alpha}e_n(T)\le\gamma_\alpha \sup_{1\le k\le n}k^{\alpha}c_k(T)
$$
for every $n\in\N$ and some constant $\gamma_\alpha$.
By homogeneity, we may assume that $c_k(T)\le k^{-\alpha}$ for $1\le k\le n$, in particular $c_1(T)=\|T\|\le1$. By monotonicity, it is also enough to prove the statement
for $n=C\,2^N$, where $N\in \N$ and $C$ is a universal natural number. Choose $\beta>\alpha>0$, $m_j=2^{N-j}$, $j=1,\dots, N$, and
$$
k_j=\lceil2^{N-j} (2/p+\beta)+1\rceil,\quad j=1,\dots,N,
$$
where $\lceil a\rceil$ is the smallest integer not less than $a$.

Then $(k_j-1)/m_j\ge 2/p+\beta$ and $\varepsilon_j:=4^{1/p}2^{-\frac{k_j-1}{m_j}}\le 2^{-\beta}.$
By Theorem \ref{thm:monster} we get
\begin{align*}
e_{k_1+\dots+k_N+1-N}(T)^q&\le 2^{-\beta Nq}+4^{q/p}\sum_{l=1}^N 2^{-\beta (l-1)q} (2^{N-l}+1)^{-\alpha q}\\
&\le 2^{-\beta Nq}+4^{q/p}2^{-\alpha Nq}2^{\beta q}\sum_{l=1}^N 2^{l(\alpha-\beta) q}\le \gamma_{\alpha,\beta}2^{-N\alpha q}.
\end{align*}
Furthermore, let $C\ge 1$ be a natural number with $C\ge 2/p+\beta+1.$ Then
\begin{align*}
1-N+\sum_{j=1}^N k_j&\le 1-N+\sum_{j=1}^N [2^{N-j}(2/p+\beta)+2]=1+N+(2/p+\beta)\sum_{j=1}^N2^{N-j}\\
&\le 1+N+(2/p+\beta)2^N\le C2^N.
\end{align*}
Putting these estimates together, we obtain
$$
e_{C2^N}(T)^q\le \gamma_{\alpha,\beta} 2^{-N\alpha q}\le c'(C2^N)^{-\alpha q},
$$
which gives the desired statement.

\begin{remark}\label{rem:const}
A detailed inspection of the proof allows to estimate the dependence of $\gamma_{\alpha}$ in \eqref{eq:orig1} on $\alpha>0$ and $0<p,q\le 1$, leading to
\begin{equation}\label{eq:gamma}
\gamma_\alpha \le 2^{2\alpha+3/q+2/p}(2/p+\alpha+1/q+2)^\alpha.
\end{equation}
In contrary to \eqref{eq:CARL}, the constant $\gamma_\alpha>0$ now depends also on the $p$-Banach space $X$ and $q$-Banach space $Y$.
Let us remark, that we did not make any efforts to optimize \eqref{eq:gamma} and its optimality is left open.
\end{remark}

\subsection{Lorentz space version}

Using standard techniques, cf. \cite[Theorem 3.1.2]{CaSt}, Carl's inequality can be easily extended to compare the Lorentz quasi-norms
of $(e_k(T))_{k\in\N}$ and $(c_k(T))_{k\in \N}$.
\begin{theo}\label{theo:Lorentz}
Let $X$ and $Y$ be quasi-Banach spaces and let $T:X\to Y$ be a bounded linear operator. Then for every $0<s\le \infty$
and every $0<t<\infty$ there exists a constant $\gamma_{s,t}$ such that for every $m\in\N$
\begin{equation}\label{eq:lorentz1}
\left(\sum_{k=1}^m k^{t/s-1}e_k(T)^t\right)^{1/t}\le \gamma_{s,t}\left(\sum_{k=1}^m k^{t/s-1}c_k(T)^t\right)^{1/t}.
\end{equation}
\end{theo}
\begin{proof} Let $\alpha>\max(1/s,1/t)$ be fixed. By Theorem \ref{theo:main} and Hardy's inequality \cite[Lemma 1.5.3]{CaSt} we get
\begin{align*}
\sum_{k=1}^m k^{t/s-1}e_k(T)^t&=\sum_{k=1}^m k^{t/s-1-\alpha t}(k^{\alpha}e_k(T))^t\le\sum_{k=1}^m k^{t/s-1-\alpha t}(\sup_{1\le l\le k}l^{\alpha}e_l(T))^t\\
&\le \gamma_{\alpha}\sum_{k=1}^m k^{t/s-1-\alpha t}(\sup_{1\le l \le k}l^{\alpha}c_l(T))^t\\
&\le \gamma_{\alpha}\sum_{k=1}^m k^{t/s-1-\alpha t}\Bigl(\sup_{1\le l \le k}\Bigl(\sum_{j=1}^l c_j(T)^{1/\alpha}\Bigr)^{\alpha}\Bigr)^t\\
&= \gamma_{\alpha}\sum_{k=1}^m k^{t/s-1-\alpha t}\Bigl(\sum_{j=1}^k c_j(T)^{1/\alpha}\Bigr)^{\alpha t}= \gamma_{\alpha}\sum_{k=1}^m k^{t/s-1}\Bigl(\frac{1}{k}\sum_{j=1}^k c_j(T)^{1/\alpha}\Bigr)^{\alpha t}\\
&\le \gamma_{\alpha,s,t}\sum_{k=1}^m k^{t/s-1}c_k(T)^t.
\end{align*}

\end{proof}

\section{Applications to optimality of sparse recovery}\label{CarlGelfand'}

Recently, the $s$-numbers were used in the area of compressed sensing \cite{CRT, D}, cf. also \cite{BCKV,FR}, to provide general lower bounds
for the performance of sparse recovery methods. In its basic setting, compressed sensing studies pairs $(A,\Delta)$
of linear measurement maps $A\in\R^{n\times N}$ and (non-linear) recovery maps $\Delta:\R^n\to\R^N$,
such that the recovery error $x-\Delta(Ax)$ is small (or even null) for $k$-sparse vectors $x\in\Sigma_k=\{x\in\R^N:\#\{i:x_i\not=0\}\le k\}$.
To allow for stability needed in applications, it is also necessary that the methods of compressed sensing
are extendable to compressible vectors, i.e. to vectors which can be very well approximated by sparse vectors.
The performance of a pair $(A,\Delta)$ in recovery of vectors from some set $K\subset \R^N$ is measured in the worst case by
$$
\varepsilon(A,\Delta,K,Y)=\sup_{x\in K}\|x-\Delta(A x)\|_Y,
$$
where $Y$ is a (quasi-)norm on $\R^N.$ The search for the optimal recovery pair is then expressed in the so-called
compressive $n$-widths
$$
E^n(K,Y)=\inf\Bigl\{\varepsilon(A,\Delta,K,Y):A\in\R^{n\times N},\Delta:\R^n\to\R^N\Bigr\}.
$$
Based on previous work in approximation theory and information based complexity \cite{M,N,Pinkus2}
it was observed in \cite{CDD, D,KT} that the compressive $n$-widths of a symmetric and subadditive set $K$ (i.e. a set $K$ with $K=-K$ and $K+K\subset aK$ for some $a>0$)
are equivalent to Gelfand numbers of $K$, which are defined as
$$
c_n(K)=\inf_{\substack{M\subset \R^{N}\\\codim\, M<n}}\sup_{x\in K\cap M}\|x\|_Y.
$$
Here, the infimum is taken over all linear subspaces $M$ of $\R^N$ with codimension smaller than $n$.
Any lower bound on Gelfand numbers of $K$ therefore immediately translates
into lower bounds on recovery errors of vectors from $K$. Especially, if an algorithm achieves the same recovery rate as
the corresponding lower bound obtained by estimates of Gelfand numbers, we know that this algorithm is asymptotically optimal.

In the frame of compressed sensing, the unit balls of $\ell_p^N$ for $0<p\le 1$ are typically used as a good model for
compressible vectors and the error of recovery is mostly measured in the Euclidean norm of $\ell_2^N$. Consequently,
Donoho \cite{D} investigated the decay of $E^n(B_p^N,\ell^N_2)$ and, consequently, the decay of Gelfand numbers of $B_p^N$ in $\ell_2^N$,
which will be denoted by $c_n(id:\ell_p^N\to \ell_2^N)$ later on.
The first estimates of these quantities for $p=1$ were obtained by Garnaev, Gluskin, and Kashin \cite{GG,G,Kashin}.
For $p<1$, the estimate from above appeared first in \cite{D} and using the approach of \cite{LGM}
it was proved also in \cite{Vyb}.

By applying \eqref{eq:CARL} to $T=id:\ell_p^N\to\ell_2^N$ and using the known results (cf. \cite{GL,Thomas,S} or \cite[Section 3.2.2]{ET}) on entropy numbers $e_n(id:\ell_p^N\to\ell_2^N)$,
Donoho obtained a lower bound for $c_n(id:\ell_p^N\to \ell_2^N)$ and, consequently, also for $E^n(B_p^N,\ell_2^N)$. The results obtained can be summarized as
\begin{equation}\label{eq:Gelfand}
c_p\min\Bigl\{1, \frac{1+\log(N/n)}{n}\Bigr\}^{1/p-1/2}\le c_n(id:\ell_p^N\to\ell_2^N)\le C_p\min\Bigl\{1, \frac{1+\log(N/n)}{n}\Bigr\}^{1/p-1/2},
\end{equation}
where $1\le n\le N$ are natural numbers and the positive numbers $c_p,C_p$ do not depend on $n$ and $N$.
The use of \eqref{eq:CARL} in the proof of the lower estimate of \eqref{eq:Gelfand} appeared for the first time in \cite{CP}
and we give a sketch of this argument in Section \ref{CarlGelfand} for readers convenience.

Unfortunately, the argument just presented contains one crucial flaw, which was overlooked by Donoho. Carl's inequality \eqref{eq:CARL}
is proven in \cite{BC} only when $X$ and $Y$ are Banach spaces. This gap was observed by H. Rauhut and T. Ullrich and the lower bound in
\eqref{eq:Gelfand} became an open problem. While attending a lecture of A. Pajor, S. Foucart saw a simple way to get the lower bound
for $c_n(id:\ell_1^N\to\ell_2^N)$. The proof avoided the use of Carl's inequality and invoked some techniques from compressed sensing.
Furthermore, the argument could be carried over to all $p\le 1$ and was then published in \cite{FPRU}, where the authors
prove the lower bound in \eqref{eq:Gelfand} for all $0<p<2$, cf. also \cite{FR}.


Although \eqref{eq:Gelfand} was proved, the question if Carl's inequality allows for an extension to quasi-Banach spaces and Gelfand numbers remained open.
Indeed, the authors of \cite{FPRU} expressed their belief that ``Carl's theorem actually fails for Gelfand widths of
general quasi-norm balls''.

\subsection{Lower bound on Gelfand numbers from Carl's inequality}\label{CarlGelfand}
In this section, we sketch the use of Carl's inequality \eqref{eq:orig1} in the proof of the lower bound in \eqref{eq:Gelfand}.
This argument appeared first in \cite{CP} and we reproduce it (with only minor modifications) here for reader's convenience.

\begin{theo}
	For $N\in\N$, $1\leq n\leq N$ and $0<p\leq q\leq 2$ it holds
\begin{equation}\label{eq:gelf1}
c_{p,q}\min\left\{1,\frac{1+\log(N/n)}{n}\right\}^{\frac{1}{p}-\frac{1}{q}}\leq c_n(id\colon\ell_p^N\to\ell_q^N)\leq C_{p,q}\min\left\{1,\frac{1+\log(N/n)}{n}\right\}^{\frac{1}{p}-\frac{1}{q}}
\end{equation}
	for some constants $c_{p,q},C_{p,q}$ not depending on $n$ and $N$.
\end{theo}
\begin{proof}
The upper bound of this inequality was already provided in \cite{Vyb}, so it only remains to prove the lower bound.
We follow the ideas of \cite[Corollary 2.6]{CP}.
By \cite{Thomas,S}, it is known that for $0<p\le q\le\infty$
\begin{equation}\label{eq:entropy}
e_n(id:\ell_p^N\to\ell_q^N)\approx\begin{cases}1&1\le n\le \log N\\\Bigl(\frac{1+\log(N/n)}{n}\Bigr)^{1/p-1/q}&\log N\le n\le N\\2^{-n/N}N^{1/q-1/p}&N\le n,\end{cases}
\end{equation}
where the constants of equivalence do not depend on the natural numbers $n$ and $N$.
Using Carl's inequality \eqref{eq:orig1}, the results on entropy numbers \eqref{eq:entropy} and the upper bound in \eqref{eq:Gelfand},
we can deduce the lower bound in \eqref{eq:Gelfand} as detailed below.

For $p=q$, \eqref{eq:gelf1} follows from $c_n(id:\ell_p^N\to\ell_p^N)=1$ in this case. We shall therefore assume that $p<q$ and for brevity
let us set $\alpha=1/p-1/q>0$. Using Carl's inequality we obtain for any natural number $n$ with $\log N\leq n\leq N$
	\begin{align}\label{eq:estimate1}
	C(n(1+\log(N/n)))^\alpha
	&\leq n^{2\alpha}e_n(id\colon\ell_p^N\to\ell_q^N)
	\leq\sup\limits_{1\leq j\leq n}j^{2\alpha}e_j(id\colon\ell_p^N\to\ell_q^N)\notag\\
	&\leq \gamma_{2\alpha}\sup\limits_{1\leq j\leq n}j^{2\alpha}c_j(id\colon\ell_p^N\to\ell_q^N)
	\end{align}
	for some constant $C$ from \eqref{eq:entropy}. For some $\lambda>1$, which we shall fix later on, let us split up this supremum into two parts to get
	\begin{align}\label{eq:sums1}
	\sup\limits_{1\leq j\leq n}j^{2\alpha}c_j(id\colon\ell_p^N\to\ell_q^N)
	\leq\sup\limits_{1\leq j\leq \lfloor\frac{n}{\lambda}\rfloor}j^{2\alpha}c_j(id\colon\ell_p^N\to\ell_q^N)
	+\sup\limits_{\lceil\frac{n}{\lambda}\rceil\leq j\leq n}j^{2\alpha}c_j(id\colon\ell_p^N\to\ell_q^N).
	\end{align}
We estimate the first summand on the right hand side by the upper bound in \eqref{eq:gelf1}
	\begin{align*}
	\sup\limits_{1\leq j\leq \lfloor\frac{n}{\lambda}\rfloor}j^{2\alpha}c_j(id\colon\ell_p^N\to\ell_q^N)
	\leq C_{p,q}\sup\limits_{1\leq j\leq \lfloor\frac{n}{\lambda}\rfloor}j^{2\alpha}\left(\frac{1+\log(N/j)}{j}\right)^\alpha
	\leq C_{p,q}\left(\frac{n(1+\log(\lambda N/n))}{\lambda}\right)^\alpha,
	\end{align*}
	where we used that the function $x\to (x\cdot(1+\log(N/x)))^\alpha$ is increasing for $1\leq x\leq N$. Using $\lambda>1$
	we end up with
	\begin{align}\label{eq:estimate2}
	\sup\limits_{1\leq j\leq \lfloor\frac{n}{\lambda}\rfloor}j^{2\alpha}c_j(id\colon\ell_p^N\to\ell_q^N)
	&\leq C_{p,q}\left(\frac{n(1+\log \lambda+\log(N/n))}{\lambda}\right)^\alpha\notag\\
	&\leq C_{p,q}\left(\frac{1+\log \lambda}{\lambda}\cdot n(1+\log(N/n))\right)^\alpha.
	\end{align}
	The second summand in \eqref{eq:sums1} can easily be estimated by monotonicity of Gelfand numbers
	\begin{equation}\label{eq:estimate3}
	\sup\limits_{\lceil\frac{n}{\lambda}\rceil\leq j\leq n}j^{2\alpha}c_j(id\colon\ell_p^N\to\ell_q^N)
	\leq n^{2\alpha} c_{\lceil\frac{n}{\lambda}\rceil}(id\colon\ell_p^N\to\ell_q^N).
	\end{equation}
	Putting the estimates \eqref{eq:estimate1}, \eqref{eq:sums1}, \eqref{eq:estimate2} and \eqref{eq:estimate3} together we arrive at
	\begin{align*}
	\gamma_{2\alpha}c_{\lceil\frac{n}{\lambda}\rceil}(id\colon\ell_p^N\to\ell_q^N)
	\geq\left(C-\gamma_{2\alpha}C_{p,q}\left(\frac{1+\log \lambda}{\lambda}\right)^\alpha\right)\left(\frac{1+\log(N/n)}{n}\right)^\alpha.
	\end{align*}
	Observing that $(1+\log \lambda)/\lambda\to0$ for $\lambda\to\infty$, there exists some $\lambda_0>1$
	such that
	\begin{align}\label{eq:lambda0}
	c_{\lceil\frac{n}{\lambda_0}\rceil}(id\colon\ell_p^N\to\ell_q^N)
	\geq C'\left(\frac{1+\log(N/n)}{n}\right)^\alpha
	\end{align}
	holds for all $n\in\N$ with  $\log N\leq n\leq N$ and some constant $C'>0$ independent of $n$ and $N$.
In case of $n<\lambda_0$, \eqref{eq:lambda0} remains true with only minor modification of the argument. The first supremum in \eqref{eq:sums1}
becomes empty and \eqref{eq:sums1} is not needed. Finally, \eqref{eq:lambda0} follows simply by \eqref{eq:estimate1} and \eqref{eq:estimate3}.

Next we prove the lower bound in \eqref{eq:gelf1} for all $k\in\N$ with $\log N \le k\le N/\lambda_0$ (if such $k$ exist). We put $n=\lfloor \lambda_0(k-1)+1\rfloor$ such that $n\le \lambda_0k\le N$
and $\lceil n/\lambda_0\rceil=k.$ By monotonicity of the function $x\to (1+\log(N/x))/x$ we therefore obtain
\begin{align*}
c_{k}(id\colon\ell_p^N\to\ell_q^N)&\geq C'\left(\frac{1+\log(N/n)}{n}\right)^\alpha
\ge C'\left(\frac{1+\log(N/(\lambda_0 k))}{\lambda_0k}\right)^\alpha\\
&\geq \frac{C'}{\lambda_0^\alpha(1+\log\lambda_0)^\alpha}\left(\frac{1+\log(N/k)}{k}\right)^\alpha=C''\left(\frac{1+\log(N/k)}{k}\right)^\alpha.
\end{align*}


It remains to prove \eqref{eq:gelf1} for $n<\log N$ and for $N/\lambda_0\leq n\leq N$. If $n<\log N$, then the claim follows from
$c_n(id\colon\ell_p^N\to \ell_q^N)\geq c_{\lceil\log N\rceil}(id\colon\ell_p^N\to \ell_q^N)$.
Finally, for $N/\lambda_0\leq n\leq N$ we use
\begin{align*}
c_n(id\colon\ell_p^N\to \ell_q^N)&\geq c_N(id\colon\ell_p^N\to \ell_q^N)=\inf_{\substack{M\subset \ell_p^N\\\codim M<N}}\sup\limits_{\substack{x\in M\\\|x\|_p\leq 1}}\|x\|_q
=\inf_{\substack{M'\subset \ell_p^N\\\dim M'=1}}\sup\limits_{\substack{x\in M'\\\|x\|_p\leq 1}}\|x\|_q\\
&=\inf_{x\in\ell_p^N, x\not=0}\frac{\|x\|_q}{\|x\|_p}=\|id\colon\ell_q^N\to\ell_p^N\|^{-1}=\Bigl(\frac{1}{N}\Bigr)^{\frac{1}{p}-\frac{1}{q}}.
\end{align*}
\end{proof}


\begin{thebibliography}{99}
\bibitem{Aoki} T. Aoki, Locally bounded linear topological spaces, Proc. Imp. Acad. Tokyo 18 (10) (1942), 588--594
\bibitem{BBP} J. Bastero, J. Bernu\'es, and A. Pe\~{n}a, An extension of Milman's reverse Brunn-Minkowski inequality, Geom. Funct. Anal. 5 (3) (1995), 572--581
\bibitem{BCKV} H. Boche, R. Calderbank, G. Kutyniok, and J. Vyb\'\i ral, A survey of compressed sensing, in: Compressed Sensing and its Applications, Birkh\"auser, Boston, 2015
\bibitem{CRT} E.J. Cand\`es, J. Romberg, and T. Tao, Robust uncertainty principles: exact signal reconstruction from highly incomplete frequency information, IEEE Trans. Inform. Theory 52 (2) (2006), 489--509
\bibitem{BC} B. Carl, Entropy numbers, $s$-numbers, and eigenvalue problems, J. Funct. Anal. 41 (3) (1981), 290--306
\bibitem{CP} B. Carl and A. Pajor, Gel'fand numbers of operators with values in a Hilbert space, Invent. Math. 94 (3) (1988), 479--504
\bibitem{CaSt} B. Carl and I. Stephani, Entropy, compactness and the approximation of operators, Cambridge University Press, 1990
\bibitem{CaTr} B. Carl and H. Triebel, Inequalities between eigenvalues, entropy numbers, and related quantities of compact operators in Banach spaces, Math. Ann. 251 (2) (1980), 129--133
\bibitem{CDD} A. Cohen, W. Dahmen, and R. A. DeVore, Compressed sensing and best k-term approximation, J. Amer. Math. Soc. 22 (1) (2009), 211--231
\bibitem{RD} R.A. DeVore, Nonlinear approximation, Acta Numer. 7 (1998), 51--150
\bibitem{DL} R.A. DeVore and G.G. Lorentz, Constructive approximation, Springer, Berlin, 1993
\bibitem{D} D.L. Donoho, Compressed sensing, IEEE Trans. Inform. Theory 52 (2006), 1289--1306
\bibitem{ET} D.E. Edmunds and H. Triebel, Function spaces, entropy numbers and differential operators, Cambridge University Press, 1996
\bibitem{FPRU} S. Foucart, A. Pajor, H. Rauhut, and T. Ullrich, The Gelfand widths of $\ell_p$-balls for $0<p\le 1$, J. Compl. 26 (6) (2010), 629--640
\bibitem{FR} S. Foucart and H. Rauhut, A mathematical introduction to compressive sensing, Birkh\"auser, Boston, 2013
\bibitem{GG} A. Garnaev and E. Gluskin, On widths of the Euclidean ball, Sov. Math. Dokl. 30 (1984), 200--204
\bibitem{GBach} M. Gerhold, Entropy, approximation and Kolmogorov numbers on quasi-Banach spaces, Bachelor thesis, Friedrich-Schiller University, Jena, 2011.
\bibitem{G} E. Gluskin, Norms of random matrices and widths of finite-dimensional sets, Math. USSR-Sb. 48  (1984), 173--182
\bibitem{GK} Y. Gordon and N. J. Kalton, Local structure theory for quasi-normed spaces, Bull. Sci. Math., 118 (1994), 441--453
\bibitem{GL} O. Gu\'edon and A.E. Litvak, Euclidean projections of a $p$-convex body, Geometric aspects of functional analysis, 
Israel Seminar (GAFA) 1996-2000, V.D. Milman and G. Schechtman, eds., Lecture Notes in Mathematics, 1745, 95--108, Springer, Berlin, 2000
\bibitem{K-HB} N. Kalton, Basic sequences in F-spaces and their applications. Proc. Edinburgh Math. Soc. 19 (2) (1974/75), 151--167
\bibitem{Kalt} N. Kalton, Quasi-Banach spaces. In: Handbook of the geometry of Banach spaces, Vol. 2, 1099--1130, North-Holland, Amsterdam, 2003
\bibitem{Kashin} B. Kashin, Diameters of some finite-dimensional sets and classes of smooth functions, Math. USSR, Izv. 11 (1977), 317--333 
\bibitem{KT} B. Kashin and V. Temlyakov, A remark on the problem of compressed sensing, Mat. Zametki 82 (6) (2007), 829--837; translation in Math. Notes 82 (5--6) (2007), 748--755
\bibitem{Koenig} H. K\"onig, Eigenvalue distribution of compact operators, Birkh\"auser, Basel, 1986
\bibitem{Thomas} T. K\"uhn, A lower estimate for entropy numbers, J. Approx. Theory 110 (1) (2001), 120--124
\bibitem{LGM} G.G. Lorentz, M.v. Golitschek, and Y. Makovoz, Constructive approximation. Advanced problems, Springer, Berlin, 1996
\bibitem{M} P. Math\'e, s-numbers in information-based complexity, J. Compl. 6 (1) (1990), 41--66
\bibitem{N} E. Novak, Optimal recovery and n-widths for convex classes of functions. J. Approx. Theory 80 (3) (1995), 390--408
\bibitem{Pietsch} A. Pietsch, Operator ideals, North-Holland Publishing Co., Amsterdam-New York, 1980
\bibitem{Pietsch2} A. Pietsch, Eigenvalues and s-numbers, Cambridge University Press, 1987
\bibitem{Pinkus} A. Pinkus, $N$-widths in approximation theory, Springer, Berlin, 1985
\bibitem{Pinkus2} A. Pinkus, $N$-widths and optimal recovery,  Proc. Symp. Appl. Math. 36. Lect. Notes AMS Short Course edition (1986), 51--66
\bibitem{Pisier} G. Pisier, The volume of convex bodies and Banach space geometry, Cambridge University Press, 1999
\bibitem{Rol} S. Rolewicz, On a certain class of linear metric spaces, Bull. Acad. Polon. Sci. S\'er. Sci. Math. Astrono. Phys. 5 (1957), 471--473
\bibitem{S} C. Sch\"utt, Entropy numbers of diagonal operators between symmetric Banach spaces, J. Approx. Theory 40 (2) (1984), 121--128
\bibitem{Vyb} J. Vyb\'\i ral, Widths of embeddings in function spaces, J. Compl. 24 (4) (2008), 545--570
\end{thebibliography}
\end{document}